\documentclass{amsart}
\usepackage{amsthm}
\usepackage{amsfonts}
\usepackage{amssymb}
\usepackage{amsmath}
\usepackage{nccmath}
\usepackage{mathtools}
\usepackage{mathrsfs}
\usepackage{setspace}
\newtheorem{theorem}{Theorem}[section]

\newtheorem{cor}{Corollary}[section]
\newtheorem{lemma}{Lemma}[section]
\theoremstyle{definition}
\newtheorem{defn}{Definition}
\theoremstyle{definition}
\newtheorem{ex}{Example}
\theoremstyle{remark}
\newtheorem{rem}{Remark}
\theoremstyle{proposition}
\newtheorem{prop}{Proposition}[section]

\begin{document}
\title[Poletsky-Stessin Hardy Spaces inside $H^{p}(\Omega)$]
 {Poletsky-Stessin Hardy Spaces on Domains Bounded by An Analytic Jordan Curve in $\mathbb{C}$ }

\author{S\.{I}bel \c{S}ah\.{I}n }

\address{Faculty of Engineering and Natural Sciences, Sabanc{\i} University}

\email{sahinsibel@sabanciuniv.edu}


\date{March 10, 2013}


\begin{abstract}
We study Poletsky-Stessin Hardy spaces that are generated by
continuous, subharmonic exhaustion functions on a domain
$\Omega\subset\mathbb{C}$, that is bounded by an analytic Jordan
curve. Different from Poletsky \& Stessin's work these exhaustion
functions are not necessarily harmonic outside of a compact set but
have finite Monge-Amp\'ere mass. We have showed that functions
belonging to Poletsky-Stessin Hardy spaces have a factorization
analogous to classical Hardy spaces and the algebra $A(\Omega)$ is
dense in these spaces as in the classical case ; however, contrary
to the classical Hardy spaces, composition operators with analytic
symbols on these Poletsky-Stessin Hardy spaces need not always be
bounded.
\end{abstract}

\maketitle
\section*{Introduction}
The aim of this paper is to study the behavior of the Hardy spaces
$H^{p}_{u}(\Omega)$ defined by Poletsky-Stessin in 2008, in a
setting where $\Omega$ is a domain in $\mathbb{C}$ containing $0$,
bounded by an analytic Jordan curve and $u$ is a continuous
subharmonic exhaustion function for $\Omega$ which has finite
Monge-Amp\'ere mass but not necessarily harmonic out of a compact
set. The main interest in Poletsky- Stessin's work (\cite{pol}) is
to study the behavior of these new spaces when the exhaustion
function, say $\varphi$ is harmonic out of a compact set (i.e the
measure $\triangle\varphi$ has compact support) whereas our main
focus will be on these Poletsky-Stessin Hardy spaces where the
associated exhaustion function $u$ has finite mass but the measure
$\triangle u$ does not have compact support. One of the main
consequences of this choice of exhaustion function is that the new
Hardy spaces and the classical ones do not coincide and so that we
have new Banach spaces to be explored inside the classical Hardy
spaces. Throughout this study we will discuss some properties of
classical Hardy spaces that can be generalized to this extent like
factorization and approximation also as far as the composition
operators on these Hardy spaces are concerned we will see that these
classes behave
different than the classical Hardy spaces.\\
Now we will briefly discuss the content of this paper. In Section 1
we will first give the definitions of classical spaces on $\Omega$,
namely Hardy and Hardy-Smirnov classes and then we will give the
necessary background information about the construction of the
Poletsky-Stessin Hardy spaces $H^{p}_{u}(\Omega)$. We will also
introduce Demailly's Lelong-Jensen formula which is one of the most
powerful tools that we use in our results. We will then compare all
those classes of holomorphic functions and will see that
Poletsky-Stessin Hardy spaces are not
always equal to the classical ones.\\
Main results of this study are given in the following sections, in
the first one we will characterize Poletsky-Stessin Hardy classes
$H^{p}_{u}(\Omega)$ through their boundary values and the
corresponding Monge-Amp\'ere boundary measure. This boundary value
characterization will enable us to prove factorization properties
analogous to classical case and next we will show the algebra
$A(\Omega)$ is dense in these spaces. Finally we will examine the
composition operators induced by holomorphic self maps and we will
see that even on the simplest of such domains , namely the unit
disc, not all composition operators are bounded contrary to the
classical Hardy space case.
\section* {Acknowledgments}
This paper is based on the first part of my PhD dissertation and I
would like to express my sincere gratitude to my advisor
Prof.Ayd{\i}n Aytuna for his valuable guidance, support and his
endless patience. This study would not have been possible without
his help. I would also like to thank Prof. Evgeny A.Poletsky for the
valuable discussions that we had during my visit at Syracuse
University and I would like to thank Dr. Nihat G\"{o}khan
G\"{o}\u{g}\"{u}\d{s} and Dr. Muhammed Ali Alan for their
suggestions and comments about this study.
\section {Poletsky-Stessin Hardy Spaces on Domains Bounded By An Analytic Jordan Curve}
In this section we will give the preliminary definitions and some
important results that we will use throughout this paper. Let
$\Omega$ be a domain in $\mathbb{C}$ containing $0$ and bounded by
an analytic Jordan curve. There are different classes of holomorphic
functions studied by various authors on domains in $\mathbb{C}$,
here we will mention first the classical Hardy space $H^{p}(\Omega)$
which is defined as follows:
\begin{equation}\label{eq:clashard}
H^{p}(\Omega)=\{f\in\mathcal{O}(\Omega)\quad|\quad |f|^{p}
\textit{has a harmonic majorant in}\quad\Omega \}
\end{equation}The second class we will mention is the Hardy-Smirnov class
$E^{p}(\Omega)$ : A holomorphic function $f$ on $\Omega$ is said to
be of class $E^{p}(\Omega)$ if there exists a sequence of
rectifiable Jordan curves $C_{1},C_{2},..$ in $\Omega$ tending to
boundary in the sense that $C_{n}$ eventually surrounds each compact
subdomain of $\Omega$ such that
\begin{equation*}
\int_{C_{n}}|f(z)|^{p}ds\leq M<\infty
\end{equation*}
In (\cite{pol}), Poletsky \& Stessin introduced new Hardy type
classes of holomorphic functions on a large class of domains in
$\mathbb{C}^{n}$ (see also (\cite{tez})) but our focus will be the
domains in $\mathbb{C}$ containing $0$ and bounded by an analytic
Jordan curve. Before defining these new classes let us first give
some preliminary definitions. Let $\varphi:\Omega\rightarrow
[-\infty,0)$ be a negative, continuous, subharmonic exhaustion
function for $\Omega$. Following Demailly define pseudoball:
\begin{equation*}\label{eq:pseudoball}
B(r)=\{z\in\Omega:\varphi(z)<r\}\quad ,\quad r\in[-\infty,0),
\end{equation*}
and pseudosphere:
\begin{equation*}\label{eq:pseudosphere}
S(r)=\{z\in\Omega:\varphi(z)=r\}\quad ,\quad  r\in[-\infty,0),
\end{equation*}
and set
\begin{equation*}
\varphi_{r}= \max\{\varphi,r\}\quad ,\quad r\in(-\infty,0).
\end{equation*} \\ In (\cite{de1}) Demailly introduced the Monge-Amp\'ere
measures in the sense of currents as :
\begin{equation*}\label{eq:mameasure}
\mu_{\varphi,r}=(dd^{c}\varphi_{r})-\chi_{\Omega\setminus
B(r)}(dd^{c}\varphi)\quad r\in(-\infty,0)
\end{equation*}\\These measures are supported on $S(r)$ and $dd^{c}\varphi=\frac{1}{4}\triangle\varphi
dzd\overline{z}$ in $\mathbb{C}$.\\Demailly had proved a general
formula in (\cite{de2}) called Lelong-Jensen Formula which is a
fundamental tool used in most of the results and in our setting it
is given as follows:
\begin{theorem}
Let $r<0$ and $\phi$ be a subharmonic function on $\Omega$ then for
any negative, continuous, subharmonic exhaustion function $u$
\begin{equation}\label{eq:leljen}
\int_{S_{u}(r)}\phi d\mu_{u,r}=\int_{B_{u}(r)}\phi
(dd^{c}u)+\int_{B_{u}(r)}(r-u)dd^{c}\phi
\end{equation}
\end{theorem}One of the main results of this study is to
characterize the Poletsky-Stessin Hardy spaces through their
boundary values and for this we also need boundary measures
introduced by Demailly in (\cite{de2}). Now let
$\varphi:\Omega\rightarrow[-\infty,0)$ be a continuous subharmonic
exhaustion for $\Omega$. And suppose that the total Monge-Amp\'ere
mass is finite, i.e.
\begin{equation}
MA(\varphi)=\int_{\Omega}(dd^{c}\varphi)<\infty
\end{equation}
Then when $r$ approaches to 0, $\mu_{r}$ converge to a positive
measure $\mu$ weak*-ly on $\Omega$ with total mass
$\int_{\Omega}(dd^{c}\varphi)$ and supported on $\partial\Omega$.
One calls the measure $\mu$ as the \textbf{Monge-Amp\'ere measure on
the boundary associated with the exhaustion $\varphi$}.\\ Now we can
introduce the Poletsky-Stessin Hardy classes which will be our main
focus throughout this study , in (\cite{pol}), Poletsky \& Stessin
gave the definition of new Hardy spaces using Monge-Amp\'ere
measures as :
\begin{defn}
$H_{\varphi}^{p}(\Omega)$, $p>0$, is the space of all holomorphic
functions $f$ on $\Omega$ such for
\begin{equation*}
\limsup_{r\to 0^{-}}\int_{S_{\varphi,}(r)}|f|^{p}
d\mu_{\varphi,r}<\infty
\end{equation*}
\end{defn}
The norm on these spaces is given by:
\begin{equation*}
\|f\|_{H_{\varphi}^{p}}=\left(\lim_{r\to
0^{-}}\int_{S_{\varphi}(r)}|f|^{p}
d\mu_{\varphi,r}\right)^{\frac{1}{p}}
\end{equation*}and with respect to these norm the spaces
$H_{\varphi}^{p}(\Omega)$ are Banach spaces (\cite{pol},pg:16).
Moreover on
these Banach spaces point evaluations are continuous.\\
Now we will compare new Hardy spaces with the classical Hardy spaces
given as in (\ref{eq:clashard}):
\begin{theorem}
Let $\Omega$ be a domain in $\mathbb{C}$ containing $0$ and bounded
by an analytic Jordan curve. Suppose $\varphi$ is a continuous,
negative, subharmonic exhaustion function for $\Omega$ such that
$\varphi$ is harmonic out of a compact set $K\subset \Omega$. Then
for a holomorphic function $f\in\mathcal{O}(\Omega)$, $f\in
H^{p}_{\varphi}(\Omega)$ if and only if $|f|^{p}$ has a harmonic
majorant.
\end{theorem}
\begin{proof}
Let $|f|^{p}$ has a harmonic majorant $u$ on $\Omega$. Then
\begin{equation}
\int_{S(r)}|f|^{p}d\mu_{\varphi,r}\leq\int_{S(r)}ud\mu_{\varphi,r}=\int_{B(r)}u(dd^{c}\varphi)
\end{equation}
by Lelong-Jensen formula and we know that $\varphi$ is harmonic out
of the compact set $K$ so
\begin{equation}
\int_{B(r)}u(dd^{c}\varphi)\leq\int_{K}u(dd^{c}\varphi)\leq
C_{K}\|u\|_{L^{\infty}(K)}
\end{equation} for some constant $C_{K}$ and this bound is
independent of $r$. Hence
\begin{equation}
\sup_{r<0}\int_{S(r)}|f|^{p}d\mu_{\varphi,r}\leq M<\infty
\end{equation}for some $M$\\
$\Rightarrow f\in H_{\varphi}^{p}(\Omega)$. For the converse,
suppose $f\in H_{\varphi}^{p}(\Omega)$ and $|f|^{p}$ has no harmonic
majorant. Then
\begin{equation}\label{eq:dblstar}
\frac{1}{2\pi}\int_{\Omega}(-g_{\Omega}(z,w))\Delta|f|^{p}=\infty
\end{equation}where $g_{\Omega}(z,w)$ is the Green function of the
domain $\Omega$ (\cite{ran}). Then from Lelong-Jensen formula
\begin{equation*}
\frac{1}{2\pi}\int_{S(r)}|f|^{p}d\mu_{\varphi,r}\geq\frac{1}{2\pi}\int_{B(r)}(r-\varphi)\Delta|f|^{p}
\end{equation*}
Note that left hand side is bounded independent from $r$ since $f\in
H_{\varphi}^{p}(\Omega)$.\\ Let us take a compact set
$F\subset\Omega$ containing the support of $\triangle\varphi$ and
$\{w\}$ such that both $\varphi$ and $g_{\Omega}(z,w)$ are bounded
on $\partial F$ and $bg_{\Omega}(z,w)\leq\varphi\leq
cg_{\Omega}(z,w)$ holds on $\partial F$ for some numbers $b,c>0$. By
the maximality of both $\varphi$ and $g_{\Omega}(z,w)$ on
$\Omega\setminus F$ this inequality holds on $\Omega\setminus F$
also hence near boundary we have
\begin{equation*}
\varphi\leq cg_{\Omega}(z,w)
\end{equation*}Then by Theorem 3.1 of (\cite{pol}, pg:11) for a positive constant $a>0$ we have
the following
\begin{equation*}
\frac{1}{2\pi}\int_{B(ar)}(ar-g_{\Omega}(z,w))\Delta|f|^{p}\leq
\frac{1}{2\pi}\int_{S(ar)}|f|^{p}d\mu_{g,r}\leq
\frac{1}{2\pi}\int_{S(r)}|f|^{p}d\mu_{\varphi,r}
\end{equation*} and as $r\rightarrow0$ by Fatou lemma we have
\begin{equation*}
\int_{\mathbb{D}}(-g_{\Omega}(z,w))\Delta|f|^{p}\leq\lim_{r\rightarrow0}\frac{1}{2\pi}\int_{B(ar)}(ar-g_{\Omega}(z,w))\Delta|f|^{p}\leq\lim_{r\rightarrow0}\frac{1}{2\pi}\int_{S(r)}|f|^{p}d\mu_{\varphi,r}
\end{equation*}
the last limit is bounded since $f\in H_{\varphi}^{p}(\Omega)$ but
the first integral goes to infinity by (\ref{eq:dblstar}). This
contradiction proves the assertion therefore $|f|^{p}$ has a
harmonic majorant.
\end{proof}
\begin{cor}
Let $\Omega$ be a domain in $\mathbb{C}$ containing $0$ and bounded
by an analytic Jordan curve. Suppose $g_{\Omega}(z,w)$ is the Green
function of $\Omega$ with the logarithmic pole at $w\in \Omega$.
Then $H^{p}(\Omega)=H^{p}_{g_{\Omega}}(\Omega)$ for $p\geq1$.
\end{cor}
\begin{proof}
The Green function $g_{\Omega}(z,w)$ is harmonic out of the compact
set $\{w\}$ by definition so by previous theorem we have $f\in
H^{p}_{g_{\Omega}}(\Omega)$ is equivalent to the condition that
$|f|^{p}$ has a harmonic majorant which is by definition means $f\in
H^{p}(\Omega)$.
\end{proof}
Now we would like to compare Poletsky-Stessin Hardy classes with the
classical Hardy and Hardy-Smirnov spaces when $\Omega$ is a domain
in $\mathbb{C}$ that is bounded by an analytic Jordan curve. First
of all on $\Omega$ we have $E^{p}(\Omega)=H^{p}(\Omega)$ by (Theorem
10.2,\cite{dur}) and using Theorem 1.1 we see that for any
exhaustion function $\varphi$ which is harmonic out of a compact set
we have $E^{p}(\Omega)=H^{p}(\Omega)=H^{p}_{\varphi}(\Omega)$.
However this need \textbf{not} be the case when our exhaustion
function $u$ has finite Monge-Amp\'ere mass yet not harmonic out of
a compact set. When an exhaustion function $u$ has finite
Monge-Amp\'ere mass then by (Theorem 3.1, \cite{pol}) we have
$H^{p}_{u}(\Omega)\subset H^{p}_{g_{\Omega}}(\Omega)=H^{p}(\Omega)$
where the last equality is given in Corollary 1.1 however by
explicitly constructing an exhaustion function $u$ on the unit disc
$\mathbb{D}$, we will show that $H^{p}_{u}(\Omega)$ need not be
equal to $H^{p}(\Omega)$ :
\begin{theorem}
There exists an exhaustion function $u$ with finite Monge-Amp\'ere
mass such that the Hardy space $H_{u}^{p}(\mathbb{D})\subsetneqq
H^{p}(\mathbb{D})$.
\end{theorem}
\begin{proof}
(Without loss of generality assume p=1) In order to prove this
result we will first construct an exhaustion function $u$ with
finite mass:\\
Let $\mathbb{D}$ be the unit disc in $\mathbb{C}$, and let $\rho$ be
the solution of the Dirichlet problem in the unit disc such that
$\triangle \rho=0$ in $\mathbb{D}$ and $\rho=f$ on
$\partial\mathbb{D}$ where $f(z)=-(1-x)^{\frac{3}{4}}$. Then define
$u=f-\rho$. We will first show that $u$ is a continuous,
subharmonic, exhaustion
function for $\mathbb{D}$:\\
$u$ is continuous:
$u(z)=f(z)+\int_{\partial\mathbb{D}}P(z,e^{i\theta})(1-\cos\theta)^{\frac{3}{4}}d\theta$
  and both parts on the right hand side are in $C^{2}(\mathbb{D})\cap
  C(\overline{\mathbb{D}})$ hence $u$ is a continuous function.\\
  $u$ is subharmonic: $u$ is a $C^{2}(\mathbb{D})$
  function and $\triangle u=\triangle(-(1-x)^{\frac{3}{4}})=(1-x)^{\frac{-5}{4}}\geq0$ hence $u$ is
  subharmonic.\\
  $u$ is an exhaustion: For this we should show that
  $A_{c}=\{x\mid u(x)<c\}$ is relatively compact in $\mathbb{D}$ for all
  $c<0$. Suppose not then there exists a sequence $x_{n}$ in $A_{c}$
  such that it has a subsequence $x_{n_{k}}$ converging to boundary
  of $\mathbb{D}$. Now $x_{n_{k}}\rightarrow x$, $|x|=1$ and $u$ is
  continuous so $u(x_{n_{k}})\rightarrow u(x)$ but $u(x_{n_{k}})<c$
  so $u(x)<c$. This contradicts the fact that $u=0$ on the boundary.
  Hence $A_{c}$ is relatively compact in $\mathbb{D}$ for all $c$
  and $u$ is an exhaustion.\\
Since we have a negative, continuous, subharmonic exhaustion
function in $\mathbb{D}$ we can define Monge-Amp\'ere measure
$\mu_{u}$ associated with $u$ and we will show that total mass
,$\|\mu_{u}\|$, of $\mu_{u}$ is finite :
\begin{equation*}
\|\mu_{u}\|=\int_{\mathbb{D}}dd^{c}u=\int_{-1}^{1}\int_{-\sqrt{(1-x)(1+x)}}^{\sqrt{(1-x)(1+x)}}(1-x)^{-\frac{5}{4}}dydx=\int_{-1}^{1}2(1+x)^{\frac{1}{2}}(1-x)^{-\frac{3}{4}}dx
\end{equation*}
\begin{equation*}
\leq2\sqrt{2}\int_{-1}^{1}(1-x)^{-\frac{3}{4}}dx
\end{equation*} say $t=1-x$ then
\begin{equation*}
=2\sqrt{2}\int_{0}^{2}\frac{1}{t^{\frac{3}{4}}}dt<\infty
\end{equation*}
Hence $\mu_{u}$ has finite mass.\\We know that for any continuous,
negative, subharmonic exhaustion function $u$ the Hardy space
$H_{u}^{p}(\mathbb{D})\subset H^{1}(\mathbb{D})$ (\cite{pol},pg:13)
but now we will show that the converse is not true by using the
$u(z)$ that we constructed above as exhaustion. Take
$F(z)=\frac{1}{(1-z)^{2q}}\quad \text{for}\quad
q<\frac{1}{2}$\\
First of all we want to show that $F(z)\in H^{1}(\mathbb{D})$ so we
will show the following growth condition is satisfied:
\begin{equation*}
\sup_{0<r<1}\int_{0}^{2\pi}|F(re^{i\theta}|d\theta<\infty\quad
\text{for}\quad 2q<1
\end{equation*}
\begin{equation*}
|F(z)|=\frac{1}{|1-re^{i\theta}|^{2q}}=\frac{1}{(1+r^{2}-2Re(re^{i\theta}))^{q}}=\frac{1}{(1+r^{2}-2r\cos(\theta))^{q}}
\end{equation*}
Now,
\begin{equation*}
\int_{0}^{2\pi}\frac{1}{(1+r^{2}-2r\cos(\theta))^{q}}d\theta=\int_{0}^{2\pi}\frac{1}{(\sin^{2}\theta+(\cos\theta-r)^{2})^{q}}d\theta\leq\int_{0}^{2\pi}\frac{1}{(\cos\theta-r)^{2q}}d\theta
\end{equation*} and as $r\rightarrow1$ this integral converges for
$2q<1$ and
\begin{equation*}
\sup_{0<r<1}\int_{0}^{2\pi}|F(re^{i\theta}|d\theta<\infty\quad
\text{for}\quad q<\frac{1}{2}
\end{equation*}
Now we will show that $F(z)\notin H^{1}_{u}(\mathbb{D})$ :
\begin{equation*}
|F(z)|=\frac{1}{|1-z|^{2q}}=\frac{1}{((1-x)^{2}+y^{2})^{q}}=((1-x)^{2}+y^{2})^{-q}
\end{equation*}from the Lelong-Jensen formula,
\begin{equation*}
\int_{S(r)}|F(z)|d\mu_{u,r}=\int_{B(r)}|F(z)|\triangle
u+\int_{B(r)}(r-u)\triangle|F(z)|
\end{equation*}second integral on the right hand side is
non-negative so
\begin{equation*}
\int_{S(r)}|F(z)|d\mu_{u,r}\geq\int_{B(r)}|F(z)|\triangle u
\end{equation*} now as $r\rightarrow0$
\begin{equation*}
\|F(z)\|_{H^{1}_{u}}\geq\lim_{r\rightarrow0}\int_{B(r)}|F(z)|\triangle
u\geq\int_{\mathbb{D}}|F(z)|\triangle u
\end{equation*} where the last inequality is due to Fatou lemma and
on $\mathbb{D}$, $y^{2}\leq1-x^{2}\leq2(1-x)$ so
\begin{equation*}
\|F(z)\|_{H^{1}_{u}}\geq\int_{\mathbb{D}}|F(z)|\triangle
u=\int_{-1}^{1}\int_{-\sqrt{(1-x)(1+x)}}^{\sqrt{(1-x)(1+x)}}((1-x)^{2}+y^{2})^{-q}(1-x)^{-\frac{5}{4}}dydx
\end{equation*}
\begin{equation*}
\geq\int_{-1}^{1}\int_{-\sqrt{(1-x)(1+x)}}^{\sqrt{(1-x)(1+x)}}((1-x)^{2}+2(1-x))^{-q}(1-x)^{-\frac{5}{4}}dydx
\end{equation*}
\begin{equation*}
\geq\int_{-1}^{1}(1-x)^{\frac{1}{2}}(1-x)^{-q}(3-x)^{-q}(1-x)^{-\frac{5}{4}}dx
\end{equation*}
\begin{equation*}
\geq\int_{-1}^{1}\frac{1}{(1-x)^{q+\frac{3}{4}}}dx
\end{equation*} so for $q>\frac{1}{4}$ this integral diverges. Hence for $\frac{1}{4}<q<\frac{1}{2}\quad F(z)\in H^{1}(\mathbb{D})$
but $F(z)\notin H^{1}_{u}(\mathbb{D})$
\end{proof}
\begin{rem}
Moreover from this result we also deduce that if $u$ is an arbitrary
exhaustion function with finite Monge-Amp\'ere mass
$H^{p}_{u}(\Omega)$'s are not always closed subspaces of
$H^{p}(\Omega)$ because as Banach spaces the inclusion
$H^{p}_{u}(\Omega)\hookrightarrow H^{p}(\Omega)$ is continuous which
can be deduced from Closed Graph Theorem and the fact that point
evaluations are continuous (\cite{pol}). However the range is not
closed because range includes all bounded functions (hence
polynomials) and polynomials are dense in $H^{p}(\Omega)$ but
$H^{p}(\Omega)\neq H^{p}_{u}(\Omega)$ in general.
\end{rem}
\section {Boundary Monge-Amp\'ere Measure \& Boundary Value Characterization of Poletsky-Stessin Hardy Spaces}
Let $\Omega$ be a domain in $\mathbb{C}$ containing $0$ and bounded
by an analytic Jordan curve and $u$ be a continuous, negative,
subharmonic exhaustion function for $\Omega$ with finite
Monge-Amp\'ere mass. In the classical Hardy space theory on the unit
disc $\mathbb{D}$ we can characterize the $H^{p}$ spaces through
their boundary values inside the $L^{p}$ spaces of the unit circle
and since we have $H^{p}_{u}(\Omega)\subset H^{p}(\Omega)$, any
holomorphic function $f\in H^{p}_{u}(\Omega)$ has the boundary value
function $f^{*}$ from the classical theory (Theorem 10,
\cite{stein}). In this section we will give an analogous
characterization of the Poletsky-Stessin Hardy spaces through these
boundary value functions and boundary Monge-Amp\'ere measure.
Moreover we will show the relation between boundary Monge-Amp\'ere
measure and Euclidean measure on the boundary $\partial\Omega$.
\begin{prop}
Let $u$ be a continuous, negative, subharmonic exhaustion function
for $\Omega$ with finite Monge-Amp\'ere mass then the boundary
Monge-Amp\'ere measure $\mu_{u}$ and the Euclidean measure on
$\partial\Omega$ are mutually absolutely continuous.
\end{prop}
\begin{proof}
Let $\varphi$ be a continuous function on $\partial\Omega$ and let
the Poisson integral of $\varphi$ be
\begin{equation*}
H(z)=\int_{\partial\Omega}P(z,\xi)\varphi(\xi)d\sigma(\xi)
\end{equation*}then by Lelong-Jensen formula we have
\begin{equation*}
\int_{\partial\Omega}\varphi
d\mu_{u}=\int_{\Omega}H(z)dd^{c}u=\int_{\Omega}\int_{\partial\Omega}P(z,\xi)\varphi(\xi)d\sigma(\xi)dd^{c}u
\end{equation*}and since $\varphi$ is a continuous function on the
boundary and $\mu_{u}$ has finite mass we can use Fubini theorem to
get
\begin{equation*}
\int_{\partial\Omega}\varphi
d\mu_{u}=\int_{\Omega}\int_{\partial\Omega}P(z,\xi)\varphi(\xi)d\sigma(\xi)dd^{c}u=\int_{\partial\Omega}\varphi(\xi)\left(\int_{\Omega}P(z,\xi)dd^{c}u(z)\right)d\sigma(\xi)
\end{equation*}
Now define
\begin{equation*}
\beta(\xi)=\int_{\Omega}P(z,\xi)dd^{c}u(z)
\end{equation*}
We will show that $\beta(\xi)$ is $d\sigma$-integrable: First we see
that $\beta(\xi)\geq0$ and
\begin{equation*}
\int_{\partial\Omega}|\beta(\xi)|d\sigma(\xi)=\int_{\partial\Omega}\beta(\xi)d\sigma(\xi)=\int_{\partial\Omega}\int_{\Omega}P(z,\xi)dd^{c}u(z)d\sigma(\xi)=\int_{\Omega}\int_{\partial\Omega}P(z,\xi)d\sigma(\xi)dd^{c}u(z)
\end{equation*}and since
\begin{equation*}
\int_{\partial\Omega}P(z,\xi)d\sigma(\xi)=1
\end{equation*}we have
\begin{equation*}
\int_{\partial\Omega}|\beta(\xi)|d\sigma(\xi)=\int_{\Omega}dd^{c}u=MA(u)<\infty
\end{equation*}Hence $\beta\in L^{1}(d\sigma)$ and we have
\begin{equation*}
\int_{\partial\Omega}\varphi
d\mu_{u}=\int_{\partial\Omega}\varphi\beta d\sigma
\end{equation*}
$\Rightarrow d\mu_{u}=\beta d\sigma$
\end{proof}
\begin{rem}
From the previous proof we see that for a fixed
$\xi\in\partial\Omega$, $\beta_{\tilde{r}}(\xi)$ which is defined as
\begin{equation*}
\beta_{\tilde{r}}(\xi)=\int_{B_{u}(\tilde{r})}P(z,\xi)dd^{c}u=\int_{S_{u}(\tilde{r})}P(z,\xi)d\mu_{u,\tilde{r}}
\end{equation*}converges to
$\beta(\xi)=\int_{\Omega}P(z,\xi)dd^{c}u(z)$ by Monotone Convergence
Theorem since $\beta_{\tilde{r}}\geq0$ for all $\tilde{r}$ and
$\beta_{\tilde{r}}$ is increasing with respect to $\tilde{r}$.
\end{rem}
Now we will give the characterization of Poletsky-Stessin Hardy
spaces $H^{p}_{u}(\Omega)$ through boundary value functions:
\begin{theorem}
Let $f\in H^{p}(\Omega)$ be a holomorphic function and $u$ be a
continuous, negative, subharmonic exhaustion function for $\Omega$.
Then $f\in H^{p}_{u}(\Omega)$ if and only if $f^{*}\in
L^{p}(d\mu_{u})$ for $1\leq p<\infty$. Moreover
$\|f^{*}\|_{L^{p}(d\mu_{u})}=\|f\|_{H^{p}_{u}(\Omega)}$.
\end{theorem}
\begin{proof}
(Without loss of generality assume p=1) Now let $f\in
H^{1}_{u}(\Omega)\subset H^{1}(\Omega)$ we want to show that
$f^{*}\in L^{1}(d\mu_{u})$. First of all
\begin{equation*}
\int_{\partial\Omega}|f^{*}(\xi)|d\mu_{u}=\int_{\partial\Omega}|f^{*}(\xi)|\left(\int_{\Omega}P(z,\xi)dd^{c}u(z)\right)d\sigma(\xi)
\end{equation*} and using Fubini-Tonelli theorem we change the order
of integration and get
\begin{equation*}
\int_{\Omega}\left(\underbrace{\int_{\partial\Omega}|f^{*}(\xi)|P(z,\xi)d\sigma(\xi)}_{H(z)}\right)dd^{c}u(z)
\end{equation*}Then the harmonic function being the Poisson integral
of $|f^{*}|$ is the least harmonic majorant of $|f|$ so by
(\cite{ran}, Theorem 4.5.4) we have
\begin{equation*}
H(z)=|f(z)|-\int_{\Omega}g_{\Omega}(w,z)dd^{c}|f(w)|
\end{equation*}where $g_{\Omega}(w,z)$ is the Green function of
$\Omega$ with the logarithmic pole at the point $z$ and by
(\cite{de2},Theorem 4.14), $g_{\Omega}(w,z)$ is continuous on
$\overline{\Omega}$ and subharmonic in $\Omega$ hence we have
\begin{equation*}
\int_{\Omega}H(z)dd^{c}u=\int_{\Omega}|f(z)|dd^{c}u-\int_{\Omega}\left(\int_{\Omega}g_{\Omega}(w,z)dd^{c}u\right)dd^{c}|f(w)
\end{equation*}Now using the boundary version of Lelong- Jensen
formula (\cite{de2},Theorem 3.3) we get
\begin{equation*}
\int_{\Omega}g_{\Omega}(w,z)dd^{c}u(z)=\int_{\Omega}u(z)dd^{c}g_{\Omega}(w,z)=u(w)
\end{equation*}therefore
\begin{equation*}
\int_{\partial\Omega}|f^{*}(\xi)|d\mu_{u}=\int_{\Omega}H(z)dd^{c}u=\int_{\Omega}|f(z)|dd^{c}u-\int_{\Omega}u(z)dd^{c}|f|=\|f\|_{H^{1}_{u}(\Omega)}<\infty
\end{equation*}so we have $f^{*}\in L^{1}(d\mu_{u})$.\\ For the
converse since $f\in H^{1}(\Omega)$ we have
\begin{equation*}
f(z)=\int_{\Omega}P(z,\xi)f^{*}(\xi)d\sigma(\xi)
\end{equation*}now
\begin{equation*}
\int_{S_{u}(r)}|f(z)|d\mu_{u,r}=\int_{S_{u}(r)}\left|\int_{\Omega}P(z,\xi)f^{*}(\xi)d\sigma(\xi)\right|d\mu_{u,r}\leq\int_{\Omega}\left(\int_{S_{u}(r)}P(z,\xi)d\mu_{u,r}\right)|f^{*}(\xi)|d\sigma(\xi)
\end{equation*}by Remark 2 we know
$\int_{S_{u}(r)}P(z,\xi)d\mu_{u,r}$ is increasing and using Monotone
Convergence Theorem as $r\rightarrow0$ we get
\begin{equation*}
\|f\|_{H^{1}_{u}(\Omega)}\leq\int_{\partial\Omega}|f^{*}(\xi)|d\mu_{u}(\xi)<\infty
\end{equation*}so $f\in H^{1}_{u}(\Omega)$
\end{proof}
\begin{rem}
The existence of boundary values for holomorphic functions depends
on conditions related to domain itself and the growth of the
functions. However throughout this study we have been exploiting the
fact that Poletsky-Stessin Hardy classes are inside the classical
ones so we already have the non tangential boundary values. In his
2011 paper (\cite{pol2}) Poletsky gave definitions for two types of
boundary values namely, strong and weak limit values for sequences
of functions in an abstract setting that involves sequences of
weakly convergent measures on increasing chains of compact sets (See
\cite{pol2} for details). If we specify his abstract setup to the
functions $f\in H^{p}_{u}(\Omega)$ together with Monge-Amp\'ere
measures $d\mu_{u,r}$ and the boundary Monge-Amp\'ere measure
$d\mu_{u}$, it is possible to prove that the boundary value function
$f^{*}$ is both weak and strong limit value for $f\in
H^{p}_{u}(\Omega)$ (Details will be given in the thesis \cite{sah})
\end{rem}
\section {Factorization}
Let $\Omega$ be a domain in $\mathbb{C}$ containing $0$ and bounded
by an analytic Jordan curve. Suppose
$\psi:\mathbb{D}\rightarrow\Omega$ is the conformal map such that
$\psi(0)=0$ and $\varphi=\psi^{-1}$. By Carath\'eodory theorem for a
domain like $\Omega$ we have $0<m\leq|\psi'|\leq M<\infty$ for some
$m,M>0$. Following (\cite{ale}) we have the following :
\begin{defn}
A holomorphic function $h$ on $\Omega$ is $\Omega$-inner if
$h\circ\psi$ is inner in the classical sense i.e. $|h\circ\psi|=1$
for almost all $\xi\in\partial\mathbb{D}$ and $\Omega$-outer if
$h\circ\psi$ is outer in the classical sense i.e.
\begin{equation*}
\log|h\circ\psi(0)|=\int_{\partial\mathbb{D}}\log|h\circ\psi|d\sigma
\end{equation*}
\end{defn}
For the classical Hardy space on the unit disc $\mathbb{D}$ we have
the following canonical factorization theorem (\cite{dur},pg:24):
\begin{theorem}
Every function $f\not\equiv0$ of class $H^{p}(\mathbb{D})$ has a
unique factorization of the form $f=BSF$ where $B$ is a Blaschke
product, $S$ is a singular inner function and $F$ is an outer
function for the class $H^{p}(\mathbb{D})$.
\end{theorem}Using this result and the conformal mapping the following is
shown in (\cite{ale}):
\begin{prop}
Every $f\in H^{p}(\Omega)$ can be factored uniquely up to a
unimodular constant as $f=IF$ where $I$ is $\Omega$-inner and $F$ is
$\Omega$-outer.
\end{prop}Inspired by this result we have the following corollary
\begin{cor}
Let $f\in H^{p}_{u}(\Omega)$ where $u$ is a continuous exhaustion
function with finite Monge-Amp\'ere mass. Then $f$ can be factored
as $f=IF$ where $I$ is $\Omega$-inner and $F$ is $\Omega$-outer.
Moreover $I\in H^{p}_{u}(\Omega)$ and $F\in H^{p}_{u}(\Omega)$.
\end{cor}
\begin{proof}
First of all since $H^{p}_{u}(\Omega)\subset H^{p}(\Omega)$ by the
above proposition it is obvious that $f$ can be factored as $f=IF$
where $I$ is $\Omega$-inner and $F$ is $\Omega$-outer. Now define
the measure $\hat{\mu}$ on $\partial\mathbb{D}$ as
$\hat{\mu}(E)=\mu_{u}(\psi(E))$ for any measurable set
$E\subset\partial\mathbb{D}$ then it is clear from the change of
variables formula that $|I\circ\psi|=1$ $d\hat{\mu}$-a.e. so
\begin{equation*}
\int_{\partial\Omega}|I|d\mu_{u}=\int_{\partial\mathbb{D}}|I\circ\psi|d\hat{\mu}=\int_{\partial\mathbb{D}}d\hat{\mu}=\int_{\partial\Omega}d\mu_{u}=MA(u)<\infty
\end{equation*}since $u$ has finite Monge-Amp\'ere mass and for the
outer part
\begin{equation*}
\int_{\partial\Omega}|F|d\mu_{u}=\int_{\partial\mathbb{D}}|F\circ\psi|d\hat{\mu}=\int_{\partial\mathbb{D}}|f^{*}\circ\psi|d\hat{\mu}=\int_{\partial\Omega}|f^{*}|d\mu_{u}<\infty
\end{equation*}since $f\in H^{p}_{u}(\Omega)$. Hence we have $I\in H^{p}_{u}(\Omega)$ and $F\in
H^{p}_{u}(\Omega)$.
\end{proof}In fact in the particular case of unit disc $\mathbb{D}$ we can say more about this factorization :
\begin{theorem}
Let $f$ be analytic function in $\mathbb{D}$ such that $f\in
H^{p}_{u}(\mathbb{D})$. Then $f$ can be factored into a Blaschke
product $B$, a singular inner function $S$ and an outer function $F$
such that $B,S,F\in H^{p}_{u}(\mathbb{D})$ for $1\leq p<\infty$.
\end{theorem}
\begin{proof}
(Without loss of generality assume p=1) Let $f\in
H^{1}_{u}(\mathbb{D})$ then $f\in H^{1}(\mathbb{D})$ so $f$ has
canonical decomposition $f=BSF$ where $B$ is a Blaschke product, $S$
is a singular inner function and $F$ is an outer function. The
exhaustion $u$ has finite mass so bounded functions belong to class
$H^{1}_{u}(\mathbb{D})$. Since $B$ and $S$ are bounded
(\cite{dur},pg:24), $B,S\in H^{1}_{u}(\mathbb{D})$. As the outer
function $F(z)$ is concerned , we know $F\in H^{1}(\mathbb{D})$ so
we have
\begin{equation*}
F(z)=\frac{1}{2\pi}\int_{0}^{2\pi}P(r,\theta-t)F^{*}(e^{it})dt
\end{equation*}
Now
\begin{equation*}
\int_{S(\widehat{r})}|F(z)|d\mu_{u,\widehat{r}}=\frac{1}{2\pi}\int_{S(\widehat{r})}\left|\int_{0}^{2\pi}P(r,\theta-t)F^{*}(e^{it})dt\right|d\mu_{u,\widehat{r}}
\end{equation*}
\begin{equation*}
\leq\frac{1}{2\pi}\int_{0}^{2\pi}\left(\int_{S(\widehat{r})}P(r,\theta-t)d\mu_{u,\widehat{r}}\right)|F^{*}(e^{it})|dt=\frac{1}{2\pi}\int_{0}^{2\pi}\left(\int_{S(\widehat{r})}P(r,\theta-t)d\mu_{u,\widehat{r}}\right)|f^{*}(e^{it})|dt
\end{equation*}and using the previous result, monotone convergence theorem and the fact that $|f^{*}|=|F^{*}|$ ($dt$)- a.e. we get as $\widehat{r}\rightarrow0$:
\begin{equation*}
\|F(z)\|_{H^{1}_{u}(\mathbb{D})}\leq\int_{\partial\mathbb{D}}|f^{*}(e^{it})|d\mu_{u}(t)<\infty
\end{equation*}
$\Rightarrow F(z)\in H^{1}_{u}(\mathbb{D})$
\end{proof}
In the classical $H^{p}$ space theory a useful tool for the proofs
is that any analytic function $f\in H^{1}(\mathbb{D})$ can be
expressed as a product of two functions , $f=gh$ where both factors
$g$ and $h\in H^{2}(\mathbb{D})$. Now we will show that there is a
similar factorization in the spaces $H^{p}_{u}(\Omega)$.
\begin{cor}
Suppose $1\leq p<\infty$, $f\in H^{p}_{u}(\Omega)$, $f\not\equiv0$.
Then there is a zero-free function $h\in H^{2}_{u}(\Omega)$ such
that $f=Ih^{\frac{2}{p}}$. In particular every $f\in
H^{1}_{u}(\Omega)$ is a product of $f=gh$ in which both factors are
in $H^{2}_{u}(\Omega)$.
\end{cor}
\begin{proof}
By the previous theorem $f/I\in H^{p}_{u}(\Omega)$. Since $f/I$ has
no zero in $\Omega$ and $\Omega$ is simply connected , there exists
$\varphi\in\mathcal{O}(\Omega)$ such that $\exp(\varphi)=f/I$.
Define $h=\exp(p\varphi/2)$ then $h\in \mathcal{O}(\Omega)$ and
$|h|^{2}=|f/I|^{p}$ and $h\in H^{2}_{u}(\Omega)$ and
$f=Ih^{\frac{2}{p}}$. To obtain $f=gh$ for $f\in H^{1}_{u}(\Omega)$
write $f=Ih^{2}$ in the form $f=(Ih)h$.
\end{proof}
\section {Approximation}
Let $A(\Omega)$ denote the algebra of holomorphic functions on
$\Omega$ which are continuous on $\partial\Omega$. We know that the
algebra of holomorphic functions $A(\Omega)$ is dense in the
classical Hardy spaces when $\Omega$ is a domain bounded by an
analytic Jordan curve and we will show an analogous approximation
result on $H^{p}_{u}(\Omega)$ where $u$ is a negative, continuous,
subharmonic exhaustion function on $\Omega$ with finite
Monge-Amp\'ere mass but before the result we should first mention
some classes of analytic functions from the classical theory on the
unit disc which will help us in the proof of approximation result:
\begin{defn}
An analytic function $f\in\mathcal{O}(\mathbb{D})$ is said to be of
class $N$ if the integrals
\begin{equation*}
\int_{0}^{2\pi}\log^{+}|f(re^{i\theta})|d\theta
\end{equation*} are bounded for $r<1$.
\end{defn}
\begin{defn}
An analytic function $f\in\mathcal{O}(\mathbb{D})$ is in class
$N^{+}$ if it has the form $f=BSF$ where $B$ is a Blaschke product,
$S$ is a singular inner function and $F$ is an outer function for
the class $N$
\end{defn}
It is clear that $N\supset N^{+}\supset H^{p}(\mathbb{D})$ for all
$p>0$.(For details see \cite{dur}). We will also use the following
result (Theorem 2.11, \cite{dur}), which plays a crucial role in our
approximation result:

\begin{theorem}
If a holomorphic function $f\in N^{+}$ and $f^{*}\in L^{p}$ for some
$p>0$, then $f\in H^{p}(\mathbb{D})$.
\end{theorem}
Now we can give the approximation result for the Poletsky-Stessin
Hardy classes $H^{p}_{u}(\Omega)$:
\begin{theorem}
The algebra $A(\Omega)$ is dense in $H^{p}_{u}(\Omega)$, $1\leq
p<\infty$.
\end{theorem}
\begin{proof}
(\textbf{Case 1: $p>1$})Let $L$ be a linear functional on
$H^{p}_{u}(\Omega)$ such that $L$ vanishes on $A(\Omega)$. Then
$L(f)=\int_{\partial\Omega}f^{*}\overline{g^{*}}d\mu_{u}$ for some
non-zero $g\in L^{q}(d\mu_{u})$ hence
$\int_{\partial\Omega}\gamma\overline{g^{*}}d\mu_{u}=0$ for all
$\gamma\in A(\Omega)$ (*). Now we need the following lemma:
\begin{lemma}
Let $\mu$ be a measure on the boundary $\partial\Omega$ of $\Omega$
which is orthogonal to $A(\Omega)$. Then $\mu$ is absolutely
continuous with respect to $d\mu_{g_{\Omega,0}}$ which is the
boundary Monge-Amp\'ere measure with respect to the Green function
with pole at $0$.
\end{lemma}
\begin{proof}
The homomorphism "evaluation at $0$" of $A(\Omega)$ has the
representing measure $d\mu_{g_{\Omega,0}}$ on $\partial\Omega$. Then
by generalized F. and M. Riesz Theorem (Theorem 7.6, \cite{game})
the singular part $\mu_{s}$ of $\mu$ with respect to
$d\mu_{g_{\Omega,0}}$ is orthogonal to $A(\Omega)$ that is
$\int_{\partial\Omega}fd\mu_{s}=0$ for all $f\in A(\Omega)$ but on a
domain like $\Omega$, $A(\Omega)$ is dense in $C(\partial\Omega)$
(Theorem 2.7, \cite{game}) so $\int_{\partial\Omega}gd\mu_{s}=0$ for
all $g\in C(\partial\Omega)$. Hence $d\mu_{s}=0$ and
$d\mu<<d\mu_{g_{\Omega,0}}$.
\end{proof}
Now by the above lemma and (*) we have
$\overline{g^{*}}d\mu_{u}<<d\mu_{g_{\Omega,0}}$ so by Radon-Nikodym
Theorem we have $\overline{g^{*}}d\mu_{u}=h^{*}d\mu_{g_{\Omega,0}}$
for some $h^{*}\in L^{1}(d\mu_{g_{\Omega,0}})$ and on a domain like
$\Omega$ we have $c_{1}d\mu_{g_{\Omega,0}}\leq d\sigma\leq
c_{2}d\mu_{g_{\Omega,0}}$ so we have $h^{*}\in L^{1}(d\sigma)$. Now
consider the function $H^{*}=h^{*}\circ\psi$ on $\partial\mathbb{D}$
then since $h^{*}\in L^{1}(d\sigma)$ we have $h^{*}\in
L^{1}(\partial\mathbb{D})$, and since $\psi(0)=0$ and
$\mu_{g_{\Omega,0}}$ is in fact the harmonic measure we have
$d\mu_{g_{\mathbb{D},0}}(e^{i\theta})=d\mu_{g_{\Omega,0}}(\psi(e^{i\theta}))$
and using $\varphi=\psi^{-1}$ we get
\begin{equation*}
\int_{\partial\mathbb{D}}e^{in\theta}(h^{*}\circ\psi(e^{i\theta}))d\theta=\int_{\partial\Omega}(\varphi(z))^{n}h^{*}(z)d\mu_{g_{\Omega,0}}(z)=\int_{\partial\Omega}(\varphi(z))^{n}\overline{g^{*}}d\mu_{u}=0
\end{equation*}for all $n$ as a consequence of (*). Hence $H^{*}$ is
the boundary value of an $H^{1}(\mathbb{D})$ function $H$ then $h$
which is defined as $H=h\circ\psi$ is in the class $E^{1}(\Omega)$
by the corollary of (Theorem 10.1, \cite{dur}). Moreover since
$E^{1}(\Omega)=H^{1}(\Omega)$ we have $h\in H^{1}(\Omega)$ and since
$\psi(0)=0$ we have $h(0)=0$.\\Now take $\alpha\in
H^{p}_{u}(\Omega)$ and consider the analytic function $\alpha h$
\begin{equation*}
\int_{\partial\Omega}|\alpha^{*}h^{*}|^{\frac{1}{2}}d\sigma\leq\left(\int_{\partial\Omega}|\alpha^{*}|d\sigma\right)^{\frac{1}{2}}\left(\int_{\partial\Omega}|h^{*}|d\sigma\right)^{\frac{1}{2}}\leq\|\alpha\|_{H^{1}}^{\frac{1}{2}}\|h\|_{H^{1}}^{\frac{1}{2}}
\end{equation*}since $h\in H^{1}(\Omega)$ and $\alpha\in H^{p}_{u}(\Omega)\subset H^{p}(\Omega)\subset
H^{1}(\Omega)$ so $\alpha h\in H^{\frac{1}{2}}(\Omega)$. On the
other hand
\begin{equation*}
\int_{\partial\Omega}|\alpha^{*}h^{*}|d\sigma\leq
c_{2}\int_{\partial\Omega}|\alpha^{*}||h^{*}|d\mu_{g_{\Omega,0}}=c_{2}\int_{\partial\Omega}|\alpha^{*}||\overline{g^{*}}|d\mu_{u}\leq\left(\int_{\partial\Omega}|\alpha^{*}|^{p}d\mu_{u}\right)^{\frac{1}{p}}\left(\int_{\partial\Omega}|g^{*}|^{q}d\mu_{u}\right)^{\frac{1}{q}}<\infty
\end{equation*}since $\alpha\in H^{p}_{u}(\Omega)$ and $g^{*}\in L^{q}(d\mu_{u})$.
Hence we have $\alpha^{*}h^{*}\in L^{1}(\partial\Omega)$. Now since
$\alpha h\in H^{\frac{1}{2}}(\Omega)$ by the corollary of (Theorem
10.1, \cite{dur}) the function $AH=\alpha
h(\psi(w))[\psi'(w)]^{2}\in H^{\frac{1}{2}}(\mathbb{D})\subset
N^{+}$ but since $\alpha^{*}h^{*}\in L^{1}(\partial\Omega)$ ,
$AH^{*}\in L^{1}(\partial\mathbb{D})$ hence again by the corollary
of (Theorem 10.1, \cite{dur}) we have $\alpha h\in H^{1}(\Omega)$.
Finally
\begin{equation*}
0=\alpha h(0)=\int_{\partial\Omega}\alpha
hd\mu_{g_{\Omega,0}}=\int_{\partial\Omega}\alpha\overline{g^{*}}d\mu_{u}=L(\alpha)
\end{equation*}for all $\alpha\in H^{p}_{u}(\Omega)$ hence
$A(\Omega)$ is dense in $H^{p}_{u}(\Omega)$.\\
(\textbf{Case 2: $p=1$}): By the previous corollary we know that if
$f\in H^{1}_{u}(\Omega)$ then we can factor it out like $f=gh$ where
$g,h\in H^{2}_{u}(\Omega)$ and from the first part of the proof we
know there exists sequences $\{g_{n}\},\{h_{n}\}\in A(\Omega)$ such
that $\{g_{n}\}\rightarrow g$ and $\{h_{n}\}\rightarrow h$ in
$H^{2}_{u}(\Omega)$. Now for $f$ consider the sequence
$\{g_{n}h_{n}\}$ then,
\begin{equation*}
\int_{\partial\Omega}|f-g_{n}h_{n}|d\mu_{u}=\int_{\partial\Omega}|gh-g_{n}h_{n}|d\mu_{u}=\int_{\partial\Omega}|gh-gh_{n}+gh_{n}-g_{n}h_{n}|d\mu_{u}
\end{equation*}
\begin{equation*}
\leq
\left(\int_{\partial\Omega}|g|^{2}d\mu_{u}\right)^{\frac{1}{2}}\left(\int_{\partial\Omega}|h-h_{n}|^{2}d\mu_{u}\right)^{\frac{1}{2}}+\left(\int_{\partial\Omega}|h_{n}|^{2}d\mu_{u}\right)^{\frac{1}{2}}\left(\int_{\partial\Omega}|g-g_{n}|^{2}d\mu_{u}\right)^{\frac{1}{2}}
\end{equation*}But right hand side goes to $0$ since $\{g_{n}\}\rightarrow g$ and $\{h_{n}\}\rightarrow h$ in
$H^{2}_{u}(\Omega)$ hence $\{g_{n}h_{n}\}\in A(\Omega)$ converges to
$f$ in $H^{1}_{u}(\Omega)$.\\
Combining these two cases we see that $A(\Omega)$ is dense in
$H^{p}_{u}(\Omega)$ for $1\leq p<\infty$.

\end{proof}

\section {Composition Operators With Analytic Symbols}
Let $\phi:\Omega\rightarrow\Omega$ be a holomorphic self map of
$\Omega$. The linear composition operator induced by the symbol
$\phi$ is defined by $C_{\phi}(f)=f\circ\phi$, $f\in
\mathcal{O}(\Omega)$. In 2003 Shapiro and Smith (\cite{shapsm})
showed that every holomorphic self map $\phi$ of $\Omega$ induces a
bounded composition operator on the classical Hardy space
$H^{p}(\Omega)$. Moreover we know that being in the class
$H^{p}_{v}(\Omega)$ where $v$ is harmonic out of a compact set is
equivalent to having a harmonic majorant hence any composition
operator on a Hardy class generated by this sort of exhaustion
function is also bounded. As a consequence of Closed Graph Theorem
continuity of a composition operator on $H^{p}_{u}(\Omega)$ is in
fact determined by whether it takes functions from
$H^{p}_{u}(\Omega)$ to $H^{p}_{u}(\Omega)$ or not. However this does
\textbf{not} always hold when exhaustion function has finite
Monge-Amp\'ere mass but not harmonic out of a compact set as we see
from the following example :
\begin{ex}
Suppose $u$ is the
exhaustion function that we constructed in Theorem 1.3 then we know
again from the proof of Theorem 1.3 that the function
$\frac{1}{(z-1)^{\frac{3}{4}}}\notin H^{1}_{u}(\mathbb{D})$. Now
consider the operator with symbol $\phi(z)=ze^{i\frac{\pi}{2}}$, and
take the function $f(z)=\frac{1}{(z-i)^{\frac{3}{4}}}$, then
\begin{equation*}
\|f\|_{H^{1}_{u}(\mathbb{D})}=\int_{\partial\mathbb{D}}\frac{1}{(\xi-i)^{\frac{3}{4}}}\beta(\xi)d\sigma(\xi)<\infty
\end{equation*}since the singularities of $\beta(\xi)$ and $f^{*}(\xi)$ do not overlap and they are both integrable functions
on the boundary so $f(z)\in H^{1}_{u}(\mathbb{D})$ but
\begin{equation*}
C_{\phi}(f)=f\circ\phi=\frac{1}{e^{i\frac{3\pi}{8}}(z-1)^{\frac{3}{4}}}
\end{equation*} and $C_{\phi}(f)\notin
H^{1}_{u}(\mathbb{D})$. Therefore not every composition operator is
bounded on Poletsky-Stessin Hardy classes even though the symbol
function is a nice and simple one like in our example, namely a
rotation.
\end{ex}
In the next result we will examine the necessary and sufficient
conditions for the composition operator $C_{\varphi}$ to be bounded
on this rather interesting space $H^{p}_{u}(\mathbb{D})$ where $u$
is the exhaustion function constructed in the proof of Theorem 1.3
and $\varphi$ is an automorphism of the unit disc:
\begin{theorem}
Let $\varphi$ be a Mobius transformation such that
$\varphi(z)=e^{i\theta}\frac{z-a}{1-\overline{a}z}$ where
$a\in\mathbb{D}$ and $u$ is the exhaustion function constructed in
the proof of Theorem 1.3. Then
the following are equivalent:\\
(i) $C_{\varphi}$ is a bounded operator on the space
$H^{p}_{u}(\mathbb{D})$ \\
(ii) There exists a constant $K>0$ such that
$\int_{E}\beta(\varphi^{*-1}(\eta))d\sigma(\eta)\leq
K\int_{E}\beta(\eta)d\sigma(\eta)$ for all measurable
$E\subset\partial\mathbb{D}$ where $d\mu_{u}=\beta
d\sigma$\\
(iii) $\varphi(1)=1$
\end{theorem}
\begin{proof}
(It is sufficient to prove the result for $p=1$)\\
$(i\Leftrightarrow ii)$ Let $f\in H^{p}_{u}(\mathbb{D})$ and
$\varphi$ be a Mobius transformation then
\begin{equation*}
\|f\circ\varphi\|_{H^{p}_{u}(\mathbb{D})}=\int_{\partial\mathbb{D}}|f^{*}\circ\varphi^{*}|\beta(\xi)d\sigma(\xi)=\int_{\partial\mathbb{D}}|f^{*}(\eta)|\beta(\varphi^{*-1}(\eta))|(\varphi^{*-1})'|d\sigma(\eta)
\end{equation*}
 Suppose $C_{\varphi}$ is bounded on
$H^{p}_{u}(\mathbb{D})$ then
$\|f\circ\varphi\|_{H^{p}_{u}(\mathbb{D})}\leq
M\|f\|_{H^{p}_{u}(\mathbb{D})}$ for all $f\in
H^{p}_{u}(\mathbb{D})$. Now since bounded functions are in
$H^{p}_{u}(\mathbb{D})$ we have $f(z)\equiv1\in
H^{p}_{u}(\mathbb{D})$ and we will write the above inequality for
$f(z)\equiv1$. Since $|(\varphi^{*-1})'|<N<\infty$ on
$\partial\mathbb{D}$ we get
\begin{equation*}
\int_{\partial\mathbb{D}}\beta(\varphi^{*-1}(\eta))|(\varphi^{*-1})'|d\sigma(\eta)\leq
N\int_{\partial\mathbb{D}}\beta(\varphi^{*-1}(\eta))d\sigma(\eta)\leq
NM\int_{\partial\mathbb{D}}\beta(\eta)d\sigma(\eta)
\end{equation*}
For the converse direction suppose that there exists a constant
$K>0$ such that
$\int_{\partial\mathbb{D}}\beta(\varphi^{*-1}(\eta))d\sigma(\eta)\leq
K\int_{\partial\mathbb{D}}\beta(\eta)d\sigma(\eta)$ for all
measurable $E\subset\partial\mathbb{D}$. Then for any characteristic
function $\chi_{E}$, $E\subset\partial\mathbb{D}$ we have
\begin{equation*}
\int_{\partial\mathbb{D}}\chi_{E}\beta(\varphi^{*-1}(\eta))|(\varphi^{*-1})'|d\sigma(\eta)=\int_{E}\beta(\varphi^{*-1}(\eta))d\sigma(\eta)\leq
K\int_{E}\beta(\eta)d\sigma(\eta)
\end{equation*}
\begin{equation*}
=K\int_{\partial\mathbb{D}}\chi_{E}\beta(\eta)d\sigma(\eta)
\end{equation*}
Hence by monotone convergence theorem for any positive integrable
function $g$ we have
\begin{equation*}
\int_{\partial\mathbb{D}}g(\eta)\beta(\varphi^{*-1}(\eta))|(\varphi^{*-1})'|d\sigma(\eta)\leq
K\int_{\partial\mathbb{D}}g(\eta)\beta(\eta)d\sigma(\eta)
\end{equation*}so
\begin{equation*}
\|f\circ\varphi\|_{H^{p}_{u}(\mathbb{D})}=\int_{\partial\mathbb{D}}|f^{*}|^{p}\beta(\varphi^{*-1}(\eta))|(\varphi^{*-1})'|d\sigma(\eta)\leq
N\int_{\partial\mathbb{D}}|f^{*}|^{p}\beta(\varphi^{*-1}(\eta))d\sigma(\eta)
\end{equation*}
\begin{equation*}
\leq
C\int_{\partial\mathbb{D}}|f^{*}|^{p}\beta(\eta)d\sigma(\eta)=C\|f\|_{H^{p}_{u}(\mathbb{D})}
\end{equation*}
Hence $C_{\varphi}$ is bounded.\\
$(i\Leftrightarrow iii)$ Suppose $C_{\varphi}$ is bounded and
$\varphi(1)\neq1$ then $\exists \xi\in\partial\mathbb{D}$ ,
$\xi\neq1$ such that $\varphi(\xi)=1$ and take the function
$f(z)=\frac{1}{(1-z)^{\frac{3}{4}}}$ then we know that $f(z)\notin
H^{1}_{u}(\mathbb{D})$. Now consider the function
$F(z)=f\circ\varphi^{-1}(z)$ then
\begin{equation*}
\|F(z)\|_{H^{1}_{u}(\mathbb{D})}=\int_{\partial\mathbb{D}}\frac{1}{|1-\varphi^{-1}(\eta)|^{\frac{3}{4}}}d\mu_{u}(\eta)
\end{equation*}
\begin{equation*}
=\int_{\partial\mathbb{D}\setminus
B_{\gamma}(1)}\frac{1}{|1-\varphi^{-1}(\eta)|^{\frac{3}{4}}}d\mu_{u}(\eta)+\int_{B_{\gamma}(1)}\frac{1}{|1-\varphi^{-1}(\eta)|^{\frac{3}{4}}}d\mu_{u}(\eta)<\infty
\end{equation*} for some $\gamma>0$. The first integral in the last
line is bounded because on $\partial\mathbb{D}\setminus
B_{\gamma}(1)$, $d\mu_{u}$ and $d\sigma$ are mutually absolutely
continuous and $\frac{1}{|1-\varphi^{-1}(\eta)|^{\frac{3}{4}}}$ is
$d\sigma$ integrable and the second integral is bounded because on
$B_{\gamma}(1)$, $\frac{1}{|1-\varphi^{-1}(\eta)|^{\frac{3}{4}}}$ is
a bounded function and hence it is $d\mu_{u}$ integrable.\\
Hence $F(z)\in H^{1}_{u}(\mathbb{D})$ but $F\circ\varphi=f\notin
H^{1}_{u}(\mathbb{D})$ which contradicts with the boundedness of
$C_{\varphi}$.\\
Suppose now $\varphi(1)=1$ from the $(i\Leftrightarrow ii)$ part of
the proof we know that if
$\frac{\beta(\varphi^{-1}(\eta))}{\beta(\eta)}<M<\infty$ then
$C_{\varphi}$ is bounded and for the case $\varphi(1)=1$ we have
$\frac{\beta(\varphi^{-1}(\eta))}{\beta(\eta)}$ bounded hence the
result follows.
\end{proof}
We can generalize this result to a slightly wider class of symbols
as follows:
\begin{theorem}
Let $\varphi:\overline{\mathbb{D}}\rightarrow \overline{\mathbb{D}}$
be a locally univalent self map of $\mathbb{D}$ such that $\varphi$
is differentiable in a neighborhood of $\overline{\mathbb{D}}$. Then
$C_{\varphi}$ is bounded on $H^{1}_{u}(\mathbb{D})$ if and only if
$\varphi(1)=1$ and $N^{\varphi}_{\beta}(\eta)\leq K\beta(\eta)$ for
some $K>0$ and all $\eta\in\partial\mathbb{D}$ where
$N^{\varphi}_{\beta}(\eta)=\sum_{j\geq1}\beta(\xi_{j}(\eta))$ and
$\{\xi_{j}(\eta)\}$ are the zeros of $\varphi(z)-\eta$.
\end{theorem}
\begin{proof}
$(\Rightarrow)$ Suppose $C_{\varphi}$ is bounded and
$\varphi(1)\neq1$ then there exists a $\xi\in\partial\mathbb{D}$
such that $\varphi(1)=\xi$ now consider the function
$f(z)=\frac{1}{(\xi-z)^{\frac{3}{4}}}$, $f(z)\in
H^{1}_{u}(\mathbb{D})$ and
\begin{equation*}
\|f\circ\varphi\|_{H^{1}_{u}(\mathbb{D})}=\int_{\partial\mathbb{D}}|f^{*}\circ\varphi|d\mu_{u}=\int_{\partial\mathbb{D}\setminus\overline{B_{\gamma}(1)}}|f^{*}\circ\varphi|d\mu_{u}+\int_{\overline{B_{\gamma}(1)}}|f^{*}\circ\varphi|d\mu_{u}
\end{equation*}where $\overline{B_{\gamma}(1)}=\partial\mathbb{D}\cap
B_{\gamma}(1)$ for some small $\gamma>0$. The first integral in the
sum is bounded since over
$\partial\mathbb{D}\setminus\overline{B_{\gamma}(1)}$,
$d\mu_{u}=Cd\sigma$ for some $C>0$ and $f\in
H^{1}_{u}(\mathbb{D})\subset H^{1}(\mathbb{D})$ so boundedness over
this region is guaranteed by classical $H^{p}$ theory but
\begin{equation*}
\int_{\overline{B_{\gamma}(1)}}|f^{*}\circ\varphi|d\mu_{u}=\int_{\overline{B_{\gamma}(1)}}\frac{1}{|\xi-\varphi|^{\frac{3}{4}}}d\mu_{u}
\end{equation*}and $\varphi$ has finite derivative near $\{1\}$ so
\begin{equation*}
\int_{\overline{B_{\gamma}(1)}}\frac{1}{|\xi-\varphi|^{\frac{3}{4}}}d\mu_{u}\geq
M\int_{\overline{B_{\gamma}(1)}}\frac{1}{|1-\eta|^{\frac{3}{4}}}d\mu_{u}\rightarrow\infty
\end{equation*}contradicting $C_{\varphi}$ being bounded. Hence
$\varphi(1)=1$.\\
The inequality $N^{\varphi}_{\beta}(\eta)\leq K\beta(\eta)$ is
trivially true for $\eta=1$ so we will consider the case where
$\eta\neq1$ and assume for a contradiction that
$N^{\varphi}_{\beta}(\eta_{0})> K\beta(\eta_{0})$ for all $K$, for
some $\eta_{0}\neq1$. Then from the definition of
$N^{\varphi}_{\beta}(\eta_{0})$ we see that
$\beta(\varphi(\eta_{0}))\rightarrow\infty$ which gives
$\varphi(\eta_{0})=1$. Then consider the function
$f(z)=\frac{1}{(\eta_{0}-z)^{\frac{3}{4}}}$ then by the same
argument above $f\circ\varphi\notin H^{1}_{u}(\mathbb{D})$
contradicting $C_{\varphi}$ being bounded hence
$N^{\varphi}_{\beta}(\eta)\leq K\beta(\eta)$ for all
$\eta\in\partial\mathbb{D}$ for some $K>0$.\\
$(\Leftarrow)$ Since $\varphi$ is locally univalent we can find a
countable collection of disjoint open arcs $\Omega_{j}$ with
$\sigma(\partial\mathbb{D}\setminus\bigcup\Omega_{j})=0$ and the
restriction of $\varphi$ to each $\Omega_{j}$ is univalent. Write
$\psi_{j}(w)$ for the inverse of $\varphi$ taking
$\varphi(\Omega_{j})$ onto $\Omega_{j}$. Then change of variables
formula gives
\begin{equation*}
\int_{\Omega_{j}}|f^{*}\circ\varphi||\varphi'(\xi)|\beta(\xi)d\sigma(\xi)=\int_{\varphi(\Omega_{j})}|f^{*}(w)|\beta(\psi_{j}(w))d\sigma(w)
\end{equation*}where $\xi=\psi_{j}(w)$. Now denoting the
characteristic function of $\varphi(\Omega_{j})$ by $\chi_{j}$ we
get
\begin{equation*}
\int_{\partial\mathbb{D}}|f^{*}\circ\varphi||\varphi'(\xi)|\beta(\xi)d\sigma(\xi)=\int_{\varphi(\partial\mathbb{D})}|f^{*}(w)|\left(\sum_{j\geq1}\chi_{j}(w)\beta(\psi_{j}(w))\right)d\sigma(w)
\end{equation*}
\begin{equation*}
=\int_{\varphi(\partial\mathbb{D})}|f^{*}(w)|\left(\sum_{j\geq1}\beta(\xi_{j}(w))\right)d\sigma(w)
\end{equation*}so
\begin{equation*}
\int_{\partial\mathbb{D}}|f^{*}\circ\varphi|d\mu_{u}\leq
M\int_{\partial\mathbb{D}}|f^{*}\circ\varphi||\varphi'(\xi)|\beta(\xi)d\sigma(\xi)
\end{equation*}
\begin{equation*}
=M\int_{\varphi(\partial\mathbb{D})}|f^{*}(w)|N^{\varphi}_{\beta}(w)d\sigma(w)\leq
KM \int_{\partial\mathbb{D}}|f^{*}(w)|d\sigma(w)
\end{equation*}hence $C_{\varphi}$ is a bounded operator.
\end{proof}

\end{document}